\documentclass[12pt,a4paper]{article}
\usepackage{amsmath}
\usepackage{amsfonts}
\usepackage{latexsym,amsthm,amscd}
\usepackage[left=2.1cm,bottom=2.7cm,top=2.7cm,right=2.1cm]{geometry}
\usepackage[colorlinks]{hyperref}
\newcounter{alphthm}
\newtheorem{thm}{Theorem}[section]
\newtheorem{lem}[thm]{Lemma}
\newtheorem{cor}{Corollary}[section]

\theoremstyle{definition}

\newtheorem{rem}{Remark}[section]

\newcommand{\be}{\begin{equation}}
\newcommand{\ee}{\end{equation}}
\newcommand{\pa}{{\partial}}

\newcommand{\g}{{\bf g}}
\newcommand{\pxi}{{\pa \over \pa x^i}}

\title{On Non-Positively Curved Homogeneous Finsler Metrics}
\author{B. Najafi and A. Tayebi}
\numberwithin{equation}{section}
\begin{document}
\maketitle

%--------------------------------------------------------------------------------------------------------------------
\begin{abstract}
In this paper, we prove two rigidity results for non-positively curved homogeneous Finsler metrics. Our first main result yields an extension of Hu-Deng's well-known result proven for the Randers metrics. Indeed, we prove that every connected homogeneous Finsler space with non-positive flag curvature and isotropic S-curvature is Riemannian or locally Minkowskian.   We extend the Szab\'{o}'s rigidity theorem for Berwald surfaces and show that homogeneous isotropic Berwald metrics with non-positive flag curvature are Riemannian or locally Minkowskian. We prove that a homogeneous $(\alpha, \beta)$-metrics  has isotropic mean Berwald curvature if and only if it has vanishing mean Berwald curvature generalizing result previously only known in the case of Randers metrics. Our second main result is to show that every homogeneous $(\alpha,\beta)$-metric with non-positive flag curvature and almost isotropic S-curvature  is Riemannian or locally Minkowskian. \\\\
{\bf {Keywords}}: Homogeneous Finsler metric, flag curvature, isotropic Berwald metric, $(\alpha,\beta)$-metric, E-curvature, S-curvature.\footnote{ 2000 Mathematics subject Classification: 53B40,53C60.}
\end{abstract}
%--------------------------------------------------------------------------------------------------------------------
\section{Introduction}
%--------------------------------------------------------------------------------------------------------------------
In Riemannian geometry, there is only one curvature derived from the Levi-Civita connection, namely the Riemannian curvature. It defines the notion of sectional curvature, which is a way to describe the curvature of Riemannian manifolds whose dimensions are greater than two. Finslerian geometry is the most natural generalization of Riemannian geometry. In the Finslerian setting, there are three curvatures derived from a Finsler connection, which are functions not merely of position but also direction. The notion of flag curvature obtained from  the hh-curvature of the Finslerian connection is a natural extension of the sectional curvature in Riemannian geometry, which tells us how curved is the Finsler manifold at a point. Alternatively, flag curvatures can be treated as Jacobi endomorphisms.

The world of Riemannian manifolds is single-colored, while the world of Finsler manifolds is like a crayon box with unlimited colors. This beauty stems from the nature of the Finsler metrics and the other two non-Riemannian curvatures of their connections. Indeed, besides the Riemannian curvature and its family, there are many interesting and essential non-Riemannian curvatures that vanish for Riemannian metrics. These non-Riemannian quantities describe the ``color'' and its rate of change over the manifold, such as the Cartan torsion ${\bf C}$, the Berwald curvature ${\bf B}$, the mean Berwald curvature ${\bf E}$, the Landsberg curvature ${\bf L}$, the mean Landsberg curvature ${\bf J}$, and the S-curvature ${\bf S}$.

One of the central problems in Finsler geometry is studying and classifying Finsler metrics of non-positive flag curvature. There are some elegant global rigidity results for the  Finsler metrics with ${\bf K} \leq 0$. The most well-known result regarding this case is the contribution of Akbar-Zadeh. He proved that every closed Finsler manifold with negative constant flag curvature ${\bf K}<0$ must be Riemannian, and every closed Finsler manifold with ${\bf K}=0$ must be locally Minkowskian \cite{AZ}. In \cite{MS}, Mo-Shen showed that every compact Finsler manifold of dimension $\geq 3$ and negative scalar curvature is a  Randers metric. The problem of Finsler metrics with non-positive flag curvature in Finsler geometry is well-studied. However, up to now, very little attention has been paid to the subject of homogeneous Finsler metrics. A Finsler manifold $(M, F)$ is said to be homogeneous if  its group of isometries  acts transitively on $M$. In \cite{HD}, Hu-Deng initiated the study
of homogeneous Finsler manifolds of positive flag curvature. They obtained  a  classification of homogeneous Randers metrics with isotropic S-curvature and positive flag curvature. In \cite{DH1}, Deng-Hu proved that a homogeneous Finsler manifold with non-positive flag curvature and negative Ricci scalar is a simply connected manifold. In \cite{DH2}, Deng-Hu classified homogeneous Finsler manifolds of positive flag curvature.

In \cite{H}, Heintze classified the class of homogeneous Riemannian manifolds with negative sectional curvature. Thus, it is natural to consider the class of non-positively curved homogeneous Finsler metrics. The first step to this problem was taken by Hu-Deng in \cite{HD}, where they showed that every connected homogeneous Randers metric with almost isotropic S-curvature and negative Ricci scalar reduces to a Riemannian metric. It turns out that every connected homogeneous Randers metric with almost
isotropic S-curvature and negative flag curvature is Riemannian. In this paper, we extend this result for general homogeneous Finsler metrics.  More precisely, we prove the following.
\begin{thm}\label{MainTHM1}
Every connected $n$-dimensional homogeneous Finsler space $(M, F)$ with  isotropic S-curvature is Riemannian or locally Minkowskian, provided that $F$ is of non-positive flag curvature for $n=2$ and is of negative flag curvature for $n>2$.
\end{thm}
It is natural to consider the non-positive curved homogeneous Finsler manifolds without the restriction imposed on their S-curvature. However, since an explicit formula of the S-curvature of the homogeneous Finsler manifolds has not been obtained, this problem is still open. An interesting question is whether Theorem \ref{MainTHM1} holds for homogeneous Finsler manifold of positive flag curvature. For now, only this can be said that every connected homogeneous Landsberg metric with positive  flag curvature and isotropic S-curvature is Riemannian.

\bigskip

\bigskip

In \cite{DH3}, Deng-Hu proved that a homogeneous Randers metric of Berwald type whose flag curvature is  non-zero everywhere must be Riemannian. Every Finsler metric of isotropic S-curvature has isotropic mean Berwald curvature or equivalently  almost isotropic S-curvature. Then by considering Theorem \ref{MainTHM1} and Deng-Hu's theorem, it is interesting to consider homogeneous Finsler manifolds  with non-positive flag curvature and almost isotropic S-curvature. However, it is challenging to investigate this question for general Finsler metrics. The first step to solve this problem is to consider it for a class of Finsler metrics that are tangible and computable. Restricting our attention to the class of $(\alpha, \beta)$-metrics, we prove the following.
\begin{thm}\label{MainTHM1.5}
Every homogeneous $(\alpha,\beta)$-metric with non-positive flag curvature, and almost  isotropic S-curvature is Riemannian or locally Minkowskian.
\end{thm}
Every Berwald metric satisfies ${\bf S}=0$. Thus, Theorem \ref{MainTHM1.5} is a natural extension of the Deng-Hu's result that proved only for Randers metrics.  In \cite{XD2}, Xu-Deng gave a complete classification of positively curved homogeneous $(\alpha, \beta)$-metrics with vanishing S-curvature. Therefore, Theorem
\ref{MainTHM1.5} can be also considered as the complement of the Deng-Hu's works.

\section{Preliminaries}\label{sectionP}
%---------------------------------------------------------------------------------------------------------------------------------------
Let $(M, F)$ be an $n$-dimensional Finsler manifold, and $TM$ be its tangent space. We denote the slit tangent space of $M$ by $TM_0$, i.e., $T_xM_0=T_xM-\{0\}$ at every $x\in M$.  The fundamental tensor $\textbf{g}_y:T_xM\times
T_xM\rightarrow \mathbb{R}$ of $F$ is defined by following
\[
\textbf{g}_{y}(u,v):={1 \over 2}\frac{\pa ^2}{\pa s \pa t} \Big[ F^2 (y+su+tv)\Big]|_{s,t=0}, \ \
u,v\in T_xM.
\]
Let $x\in M$ and $F_x:=F|_{T_xM}$. To measure the
non-Euclidean feature of $F_x$, define ${\bf C}_y:T_xM\times T_xM\times
T_xM\rightarrow \mathbb{R}$ by
\[
{\bf C}_{y}(u,v,w):={1 \over 2} \frac{d}{dt}\Big[\textbf{g}_{y+tw}(u,v)
\Big]|_{t=0}, \ \ u,v,w\in T_xM.
\]
The family ${\bf C}:=\{{\bf C}_y\}_{y\in TM_0}$ is called the Cartan torsion. By definition, ${\bf C}_y$ is a symmetric trilinear form on $T_xM$. It is well known that ${\bf{C}}=0$ if and only if $F$ is Riemannian.

Let $(M, F)$ be a Finsler manifold. For  $y\in T_x M_0$, define ${\bf I}_y:T_xM\rightarrow \mathbb{R}$
by
\[
{\bf I}_y(u)=\sum^n_{i=1}g^{ij}(y) {\bf C}_y(u, \partial_i, \partial_j),
\]
where $\{\partial_i\}$ is a basis for $T_xM$ at $x\in M$. The family
${\bf I}:=\{{\bf I}_y\}_{y\in TM_0}$ is called the mean Cartan torsion. By definition, ${\bf I}_y(u):=I_i(y)u^i$, where $I_i:=g^{jk}C_{ijk}$. By Deicke's theorem, every positive-definite Finsler metric
 $F$ is Riemannian if and only if ${\bf I}=0$.

\bigskip

Given a Finsler manifold $(M, F)$, then a global vector field ${\bf G}$ is induced by $F$ on $TM_0$, and in a standard coordinate $(x^i,y^i)$ for $TM_0$ is given by ${\bf G}=y^i {{\partial} / {\partial x^i}}-2G^i(x,y){{\partial}/ {\partial y^i}}$, where $G^i=G^i(x, y)$ are scalar functions on $TM_0$ given by
\be
G^i:=\frac{1}{4}g^{ij}\Bigg\{\frac{\partial^2[F^2]}{\partial x^k
\partial y^j}y^k-\frac{\partial[F^2]}{\partial x^j}\Bigg\},\ \
y\in T_xM.\label{G}
\ee
The vector field ${\bf G}$ is called the spray associated with $(M, F)$.

\bigskip
A natural volume form $dV_F = \sigma_F(x) dx^1 \cdots
dx^n$ of a Finsler metric $F$ on an $n$-dimensional manifold $M$ is defined by
\begin{equation}\label{sigma}
\sigma_F(x) := {{\rm Vol} (\Bbb B^n)
\over {\rm Vol} \Big \{ (y^i)\in \mathbb{R}^n \ \Big | \ F \big ( y^i
\pxi|_x \big ) < 1 \Big \} },
\end{equation}
where $\Bbb B^n=\{y\in\mathbb{R}^n|\,\, |y|<1\}$. The S-curvature is defined by
\[
 {\bf S}({\bf y}) := {\pa G^i\over \pa y^i}(x,y) - y^i {\pa \over \pa x^i}
\Big [ \ln \sigma_F (x)\Big ],
\]
where ${\bf y}=y^i\partial/\partial x^i|_x\in T_xM$. We say $F$ has almost isotropic S-curvature if
\[
{\bf S}= (n+1) c F+\eta
\]
where $c= c(x)$ is a scalar function  and $\eta=\eta_i(x)y^i$ is a closed 1-form on $M$. If $\eta=0$, then $F$ has isotropic S-curvature.

\newpage

For $y \in T_xM_0$, define ${\bf B}_y:T_xM\times T_xM \times T_xM\rightarrow T_xM$ by ${\bf B}_y(u, v, w):=B^i_{\ jkl}(y)u^jv^kw^l{{\partial } \over {\partial x^i}}|_x$ where
\[
B^i_{\ jkl}:={{\partial^3 G^i} \over {\partial y^j \partial y^k \partial y^l}}.
\]
The quantity $\bf B$ is called the Berwald curvature of the Finsler metric $F$. We call a Finsler metric $F$ a Berwald metric,  if ${\bf{B}}=0$. Moreover, $F$ is called of isotropic Berwald curvature if
\begin{eqnarray}\label{IBCurve}
\nonumber{\bf B}_y(u,v,w)&=& cF^{-1}\Big\{{\bf h}(u,v)\Big (w-\textbf{g}_y(w,\ell)\ell\Big)+{\bf h}(v, w)\Big(u-\textbf{g}_y(u,\ell)\ell\Big)
\\ &&\quad\quad \ \ \ \ \ \ \ \ \ +{\bf h}(w, u)\Big(v-\textbf{g}_y(v,\ell)\ell\Big)+2F{\bf C}_y(u, v, w) \ell\Big\}.
\end{eqnarray}
where $c= c(x)$ is a scalar function on $M$, and ${\bf h}=h_{ij}dx^i\otimes dx^j$ is the angular metric.

\bigskip
Define the mean of Berwald curvature by ${\bf E}_y:T_xM\times T_xM \rightarrow \mathbb{R}$, where
\be
{\bf E}_y (u, v) := {1\over 2} \sum_{i=1}^n g^{ij}(y) g_y \Big ( {\bf B}_y (u, v, e_i ) , e_j \Big ).
\ee
The family ${\bf E}=\{ {\bf E}_y \}_{y\in TM\setminus\{0\}}$ is called the {\it mean Berwald curvature} or {\it E-curvature}.
In a local coordinates,  ${\bf E}_y(u, v):=E_{ij}(y)u^iv^j$, where
\[
E_{ij}:=\frac{1}{2}B^m_{\ mij}.
\]
By definition, ${\bf E}_y(u, v)$ is symmetric in $u$ and  $v$ and we have ${\bf E}_y(y, v)=0$. $\bf E$ is called the mean
Berwald curvature. $F$ is called a weakly Berwald metric if ${\bf{E}}=0$. $F$ is called of isotropic E-curvature, if
\[
{\bf E}= \frac{n+1}{2} c F^{-1}{\bf h},
\]
where $c= c(x)$ is a scalar function  on $M$.

\bigskip

For $y\in T_xM$, define the Landsberg curvature ${\bf L}_y:T_xM\times T_xM \times T_xM\rightarrow \mathbb{R}$ by
\[
{\bf L}_y(u, v,w):=-\frac{1}{2}{\bf g}_y\big({\bf B}_y(u, v, w), y\big).
\]
$F$ is called a Landsberg metric if ${\bf L}_y=0$. By definition, every Berwald metric is a Landsberg metric.

\bigskip
Let $(M, F)$ be a Finsler manifold. For  $y\in T_x M_0$, define ${\bf J}_y:T_xM\rightarrow \mathbb{R}$
by
\[
{\bf J}_y(u)=\sum^n_{i=1}g^{ij}(y) {\bf L}_y(u, \partial_i, \partial_j).
\]
The quntity $\bf J$ is called the mean Landsberg curvature or J-curvature of Finsler metric $F$.
A Finsler metric $F$ is called a weakly Landsberg metric if ${\bf J}_y=0$. By definition, every Landsberg metric is a weakly Landsberg metric.
Mean Landsberg curvature can also be defined as following
\[
J_i: = y^m {\pa I_i \over \pa x^m} -I_m {\pa G^m\over \pa y^i} - 2 G^m {\pa I_i \over \pa y^m}.
\]
By definition, we get
\begin{eqnarray*}
{\bf J}_y (u):= {d\over dt} \Big [ {\bf I}_{\dot{\sigma}(t) } \big ( U(t) \big )\Big ]_{t=0},
\end{eqnarray*}
where $y\in T_xM$, $\sigma=\sigma(t)$ is the geodesic with $\sigma(0)=x$, $\dot{\sigma}(0)=y$, and $U(t)$ is a  linearly parallel vector field along $\sigma$ with
$U(0)=u$. The mean Landsberg curvature ${\bf J}_y$ is the rate of change of ${\bf I}_y$ along geodesics
for any $y\in T_xM_0$.

\bigskip
For an arbitrary  non-zero vector $y \in T_xM_{0}$, the Riemann curvature is a  linear
transformation $\textbf{R}_y: T_xM \rightarrow T_xM$ with homogeneity ${\bf R}_{\lambda y}=\lambda^2 {\bf R}_y$,
$\forall \lambda>0$, which is defined by
$\textbf{R}_y(u):=R^i_{k}(y)u^k {\partial / {\partial x^i}}$, where
\be
R^i_{k}(y)=2{\partial G^i \over {\partial x^k}}-{\partial^2 G^i \over
{{\partial x^j}{\partial y^k}}}y^j+2G^j{\partial^2 G^i \over
{{\partial y^j}{\partial y^k}}}-{\partial G^i \over {\partial
y^j}}{\partial G^j \over {\partial y^k}}.\label{Riemannx}
\ee
The family $\textbf{R}:=\{\textbf{R}_y\}_{y\in TM_0}$ is called the Riemann curvature of the Finsler manifold $(M, F)$.

\bigskip

For a flag $P:={\rm span}\{y, u\} \subset T_xM$ with flagpole $y$, the flag curvature ${\bf
K}={\bf K}(x, y, P)$ is defined by
\be
{\bf K}(x, y, P):= {\g_y \big(u, {\bf R}_y(u)\big) \over \g_y(y, y) \g_y(u,u)
-\g_y(y, u)^2 }.\label{FC0}
\ee
The flag curvature ${\bf K}(x, y, P)$ is a function of tangent planes $P={\rm span}\{ y, v\}\subset T_xM$. This quantity tells us how curved space is at a point. A Finsler metric $F$ is of scalar flag curvature if $\textbf{K} = \textbf{K}(x, y, P)$ is independent of flags $P$ containing $y\in T_xM_0$.

\section{Proof of Theorem \ref{MainTHM1}}
In this section, we are going to prove Theorem \ref{MainTHM1}. For this aim, we need the following.

\bigskip

\begin{lem}\label{alaki2}
Let $\Psi:\mathbb{R}\to\mathbb{R}$ be a bounded smooth function. If $\Psi''$ is a non-negative function, then $\Psi$ is a constant function.
\end{lem}
\begin{proof}
Suppose that $\Psi$ is not a constant function. Then, there exists an interval $(a,b)$ such that $\Psi(a)\neq \Psi(b)$. Without loss of generality, we suppose that $\Psi(a)< \Psi(b)$. By the mean value theorem, there exists $t_0\in(a,b)$ such that
\[
\Psi'(t_0)=\frac{ \Psi(b)-\Psi(a)}{b-a}>0.
\]
Since $\Psi''$ is a non-negative function, for every $t>t_0$ we have $\Psi'(t)\geq\Psi'(t_0)>0$. For any natural number $n$, there exists $t_n>t_0$ such that
\[
\frac{ \Psi(t_0+n)-\Psi(t_0)}{n}=\Psi'(t_n)\geq\Psi'(t_0),
\]
which implies that 
\[
\Psi(t_0+n)\geq \Psi(t_0)+n\Psi'(t_0), \ \  \ \  \forall n\in \mathbb{N}.
\]
This yields that $\Psi$ is not bounded, which is a contradiction. This completes the proof.
\end{proof}

In \cite{LR}, Latifi-Razavi proved that every  homogeneous Finsler manifold is forward geodesically complete. In \cite{TN1}, Tayebi-Najafi improved their result and proved the following.
\begin{lem}{\rm (\cite{TN2})}\label{lem1}
Every homogeneous Finsler manifold is complete.
\end{lem}

\bigskip
By definition, every two points of a homogeneous Finsler manifold $(M, F)$ map to each other under an isometry. This causes the norm of  an invariant  tensor  under the isometries of  a homogeneous Finsler manifold is a constant function on $M$, and consequently, it has a bounded norm. Then, we conclude the following.
\begin{lem}{\rm (\cite{TN1})}\label{lem2}
Let $(M, F)$ be a homogeneous Finsler manifold. Then, every invariant tensor under the isometries of $F$ has a bounded norm with respect to $F$.
\end{lem}

\bigskip
Here, we prove  a result that is the crucial lemma throughout the paper.

\begin{lem}\label{CLem}
Let $(M, F)$ be a homogeneous Finsler manifold of  isotropic S-curvature. Suppose that $F$ has non-positive flag curvature. Then, $F$ is a weakly Landsberg metric.
\end{lem}
\begin{proof}
Combining formula (5) in \cite{MS} and formula (21) in \cite{CMS}, one can obtain
\be
J_{k|m}y^m + I_mR^m_{\ k} ={\bf S}_{\cdot k|m}y^m - {\bf S}_{| k}, \label{EIILJ}
\ee
which is equivalent to
\be
I^i_{\ | p|q} y^py^q + R^i_{\ m}I^m = g^{ik} \Big \{ {\bf S}_{\cdot k|m}y^m - {\bf S}_{|k} \Big \}. \label{EIILJ*}
\ee
It is proved that every homogeneous Finsler metric of  isotropic S-curvature ${\bf S}=(n+1)cF$ has vanishing S-curvature ${\bf S}=0$ (Corollary 4.3. in \cite{HD2}).  Thus,  \eqref{EIILJ*} and vanishing of the S-curvature imply that $I^i_{\ | p|q} y^py^q + R^i_{\ m}I^m =0$ or equivalently
\be
I^i_{\ | p|q} y^py^q=- R^i_{\ m}I^m. \label{JIK*}
\ee
Suppose that $F$ is a non-Riemannian metric. Then by famous Deicke's theorem $||{\bf I}||_{(x,y)}\neq 0$ at some non-zero tangent vectors $y\in T_xM$. Suppose $y\in T_xM_0$ is such a vector and let $\sigma=\sigma(t)$ be the unit geodesic of $F$ with $\sigma(0)=x$ and $\sigma'(0)=y$. Let us define
\[
\Psi(t):=||{\bf I}||_{(\sigma(t),\sigma'(t))}
\]
and suppose that $ (a,b)$ is the maximal interval on which $\Psi(t)$ is positive.

Taking a horizontal derivation of $\Psi(t)^2=I_i\big(\sigma(t),\sigma'(t)\big)I^i\big(\sigma(t),\sigma'(t)\big)$ along geodesics and using the Chaushi-Schwarz inequality, we get
\begin{eqnarray}
\Psi(t) \Psi'(t)=I_i{\big(\sigma(t), \sigma'(t)\big)}J^i{\big(\sigma(t),\sigma'(t)\big)}\leq ||{\bf I}||_{(\sigma(t),\sigma'(t))}\ ||{\bf J}||_{(\sigma(t), \sigma'(t))}=\Psi(t) \ ||{\bf J}||_{(\sigma(t), \sigma'(t))} \ \ \
\end{eqnarray}
which implies
\begin{eqnarray}
 \Psi'(t)\leq \ ||{\bf J}||_{(\sigma(t),\sigma'(t))}.\label{s11}
\end{eqnarray}
By assumption, ${\bf K}\leq 0$ and (\ref{s11}), we have
\begin{eqnarray}
[ \Psi^2]''(t)&\overset{ \eqref{JIK*}}{=}&2\Big[- R^i_{\ m}{\big(\sigma(t),\sigma'(t)\big)}I^m{\big(\sigma(t),\sigma'(t)\big)}I_i{\big(\sigma(t),\sigma'(t)\big)}+||{\bf J}||_{(\sigma(t),\sigma'(t))}^2\Big]\nonumber\\
&\geq&2||{\bf J}||_{(\sigma(t),\sigma'(t))}^2\geq  2\Psi'(t)^2.\label{aki}
\end{eqnarray}
We obtain that $\Psi''(t) \geq 0$. Using Lemma \ref{alaki2}, we conclude that  $\Psi$ is a constant function and then $\Psi'(t) \equiv 0$. By \eqref{aki} and ${\bf K}\leq 0$, we get $||{\bf J}||_{(\sigma(t),\sigma'(t))}=0$. Due to the arbitrariness of the non-zero vector $y\in T_xM$, it follows that $F$ is a weakly Landsberg metric.
\end{proof}

\noindent{\bf Proof of Theorem \ref{MainTHM1}:} We first deal with Finsler surfaces. The special and useful Berwald frame was introduced and developed by Berwald \cite{B}. Let $(M, F)$ be a two-dimensional Finsler manifold. One can define a local field of orthonormal frame $(\ell^i,m^i)$ called the Berwald frame, where  $\ell^i=y^i/F$, $m^i$ is the unit vector with $\ell_i m^i=0$,   $\ell_i=g_{ij}\ell^i$ and $g_{ij}$ is  defined by $g_{ij}=\ell_i\ell_j+m_im_j$. The Berwald curvature of Finsler surfaces is given by
\begin{equation}\label{1}
B^i_{\ jkl}=F^{-1}(-2I_{,1}\ell^i+I_2m^i) m_j m_k m_l,
\end{equation}
where $I=I(x, y)$ is 0-homogeneous function called the main scalar of Finsler metric and $I_2=I_{,2}+I_{,1|2}$ (see page 689 in \cite{An}). In \cite{TP}, it is  proved that the Berwald curvature of a Finsler surface can be written as follows
\begin{equation}\label{B0}
B^i_{\ jkl}=\mu C_{jkl}\ell^i+\lambda(h^i_j h_{kl}+h^i_k h_{jl}+h^i_l h_{jk}),
\end{equation}
where $\mu:=-{2I_{,1}}/{I}$ and $\lambda:={I_2}/({3F})$ are homogeneous functions on $TM$ of degrees 0 and -1 with respect to $y$, respectively. Taking a trace of (\ref{B0}) yields
\be\label{B2}
E_{jk}=\frac{3}{2}\lambda h_{jk}.
\ee
Contracting (\ref{B0}) with $y_i$ implies that
\be\label{B3b}
L_{jkl}+\frac{\mu}{2}FC_{jkl}=0.
\ee
Multiplying \eqref{B0} with $g^{jk}$ yields
\be\label{B3b}
J_l+\frac{\mu}{2}FI_l=0.
\ee
By Lemma \ref{CLem}, $F$ satisfies ${\bf J}=0$. By putting it in  \eqref{B3b}, we conclude that $F$ is Riemannian or $\mu=0$. On the other hand, by assumption, we have ${\bf S}=0$ and then ${\bf E}=0$. Substituting it in \eqref{B2} gives us  $\lambda=0$. Plugging $\mu=\lambda=0$ in \eqref{B0} implies that $F$ is a Berwald metric. In \cite{Sz}, Szab\'{o} proved that every connected Berwald surface is Riemannian or locally Minkowskian. This completes the proof for 2-dimensional Finsler spaces.

Now, let us consider a Finsler manifold $(M,F)$ with dimension greater than two.  Suppose that ${\bf K}<0$.  In this case, let us consider the scalar function $f:TM\to \mathbb{R}$ defined by  $f(x,y):=F^2(x,y)g_y({\bf I}_y,  {\bf I}_y)$. The scalar function $f$ is a homogeneous function  of degree zero on $TM_0$.

It is known that the Lie derivative of $F$ along the spray of  $F$ vanishes, i.e.,  $\pounds_{{\bf G}}(F)=0$. Therefore, we have
\begin{equation}
\pounds_{{\bf G}}(f)=f_{|s}y^s=F^2I^iI_{i|s}y^s+F^2I^i_{\,\,|s}y^sI_i=2F^2J^iI_i=0,\label{new}
\end{equation}
where we have used ${\bf J}=0$. The relation (\ref{new})  means that $f$ is constant along geodesics of $F$. Using a Ricci identity given in \cite{MS}, we get
\begin{equation}
f_{,p}R^p_{\ i}+f_{,i|p|q}y^py^q=0.\label{new0}
\end{equation}
Let $\phi$ be a local isometry of $F$. It is easy to see that $f$ is invariant under $\phi$, i.e., in a standard local coordinates, we have
\be
f(x^i, y^i) = f\Big(\phi^i(x), y^j\frac{\partial \phi^i}{\partial x^j}\Big).\label{X1}
\ee
Let us put
\[
\bar{x}^i = \phi^i(x), \ \ \  \bar{y}^i =y^j \frac{\partial \phi^i}{\partial x^j}
\]
Thus
\be
f(x^i, y^i)=f(\bar{x}^i, \bar{y}^i)
\ee
Let us define $\phi^i_j:={\partial \phi^i}/{\partial x^j}$. Since $\phi$ is an isometry,  the matrix $(\phi^i_j)$ is invertible. Put $(\psi^i_j):=(\phi^i_j)^{-1}$. We have
\[
\frac{\partial \phi^i}{\partial y^j}=0, \ \ \ \ \frac{\partial \psi^i}{\partial y^j}=0.
\]
Put
\[
g_{ij}:=\frac{1}{2}\frac{\partial ^2 F^2}{\partial y^i \partial y^j}, \ \ \ \  \bar{g}_{ij}:=\frac{1}{2}\frac{\partial ^2 F^2}{\partial \bar{y}^i \partial \bar{y}^j}.
\]
Thus
\begin{eqnarray}
g_{ij}=\bar{g}_{rs}\phi^r_{i}\phi^s_{j}.\label{Xg1}
\end{eqnarray}
Equivalently
\begin{eqnarray}
\bar{g}_{ij}=g_{rs}\psi^r_{i}\psi^s_{j}.\label{Xg2}
\end{eqnarray}
It follows that
\[
\bar{g}^{ij}=g^{pq}\phi^i_p\phi^j_q.
\]
By definition, we have $\bar{y}^i=y^j\phi^i_j$. Thus
\be
\bar{y}_i:=\frac{\partial F}{\partial \bar{y}^i}=y_j\psi^j_i.
\ee
The following holds
\be
\frac{\partial f}{\partial y^i}=\phi^r_i   \frac{\partial f}{\partial \bar{y}^r}\label{T1}
\ee
This means that the tensor filed with components ${\partial f}/{\partial y^i}$ is invariant under isometries of $F$. Thus, it
has bounded norm with respect to $F$. This means that the scalar function
\[
\tilde{f}:=F^2g^{ij}\frac{\partial f}{\partial y^i}\frac{\partial f}{\partial y^j}
\]
is invariant under the isometries of $F$. The scalar function $\tilde{f}$ is a
homogeneous function  of degree zero on $TM_0$.  Let $\sigma: \mathbb{R}\rightarrow M$ be an arbitrary unit speed geodesic of $F$. To avoid clutter, we use the abbreviation  $\tilde{f}(t):=\tilde{f}(\sigma(t),\dot{\sigma}(t))$. Our argument shows that $\tilde{f}(t):\mathbb{R}\to\mathbb{R}$ is a bounded smooth function.
It is easy to see that
\begin{equation}
\tilde{f}''(t)=2f_{,i|p|q}\dot{\sigma}^p\dot{\sigma}^qg^{ij}+2f_{,i|p}f_{,j|q}\dot{\sigma}^p\dot{\sigma}^qg^{ij}.\label{new10}
\end{equation}
Plugging (\ref{new0}) into (\ref{new10}), we get
\begin{equation}
\tilde{f}''(t)=-2R^i_{\ k}f_{,i}f_{,j}g^{jk}+2f_{,i|p}f_{,j|q}\dot{\sigma}^p\dot{\sigma}^qg^{ij}.\label{new11}
\end{equation}
Since $F$ has non-positive flag curvature, then we have $\tilde{f}''(t)\geq 0$. It follows from Lemma \ref{alaki2} that $\tilde{f}$ is a constant function. Thus, $\tilde{f}'$ is zero function, and consequently,
$\tilde{f}''=0$.  It follows form (\ref{new11})
\begin{equation}
R^i_{\ k}f_{.i}f_{.j}g^{jk}=f_{,i|p}f_{,j|q}\dot{\sigma}^p\dot{\sigma}^qg^{ij}=0.\label{new13}
\end{equation}
Negatively curved condition and the arbitrariness of the geodesic $\sigma$ imply that $f_{,i}=0$. It means that $f$ is a function of position.
From (\ref{new}), we get
\[
\frac{\partial f}{\partial x^i}y^i=0.
\]
Thus, ${\partial f}/{\partial x^i}=0$ and as a result
$f$ is a constant. We recall that $I_i={\partial \tau}/{\partial y^i}$. For a fixed point $x_0\in M$, the distorsion attains its extremum on
indicatrix of $F$ at $x_0$. At this point $f$ vanishes, and constancy of $f$ implies that $f=0$. The proof follows from  Deicke's theorem.
\qed

\bigskip

\begin{rem}
With the notation of Theorem \ref{MainTHM1}, suppose that $F$ has vanishing flag curvature.  According to \cite{AZ}, any positively complete Finsler metric with zero flag curvature  must be locally Minkowskian if the first and second Cartan torsions are bounded. For the homogeneous Finsler metrics, the first and second Cartan torsions are bounded. Then in this case, $F$ reduces to a locally Minkowskian metric.
\end{rem}
In \cite{HD}, Hu-Deng proved that every homogeneous Randers metric of isotropic S-curvature and negative flag curvature is Riemmanian. Then Theorem \ref{MainTHM1} is an extension of their result.

\bigskip
In \cite{TN2}, the authors proved that every homogeneous isotropic Berwald metric on a manifold $M$ of dimension $n\geq 3$ is either a Berwald metric or a
Randers metric of Berwald type. Here, we prove the following.
\begin{cor}\label{cor1}
Every connected homogeneous isotropic Berwald manifold with non-positive flag curvature  is Riemannian or locally Minkowskian.
\end{cor}
\begin{proof}
In \cite{TR}, it is proved that isotropic Berwald metric has isotropic S-curvature. By Theorem \ref{MainTHM1}, we get the proof.
\end{proof}

Also, according to the Szab\'{o}'s theorem, every Berwald manifold is Riemmanian (if ${\bf K}\neq 0$) or locally Minkowsian (if ${\bf K}=0$). Thus the Corollary \ref{cor1} is an extension of his result for homogeneous isotropic Berwald manifolds.

\bigskip

\begin{cor}\label{cor2}
Every connected homogeneous Einstein manifold with non-positive flag curvature is Riemannian or locally Minkowskian.
\end{cor}
\begin{proof}
In \cite{BR}, Bao-Robles proved that every Einstein Finsler metric has constant S-curvature. By Theorem \ref{MainTHM1}, we get the proof.
\end{proof}

\bigskip

Finsler metrics of sectional flag curvature are those Finsler metrics whose flag curvatures ${\bf K}(x, y, P)$ only depend on sections $P$.  Trivial examples are Riemannian metrics and isotropically or constantly curved Finsler metrics. In \cite{HS},  Huang-Shen proved that there is no non-trivial Finsler metric of sectional flag curvature.

\begin{cor}
Every homogeneous Finsler metric of non-positive sectional flag curvature and isotropic S-curvature is Riemannian or locally Minkowskian.
\end{cor}
\begin{proof}
Every  homogeneous Finsler metric of non-positive sectional flag curvature is of non-positive scalar flag curvature \cite{HS}. By Theorem \ref{MainTHM1}, we get the proof.
\end{proof}
%\bigskip In \cite{BCS}, Bao-Chern-Shen proved that every compact Landsberg surfaces with flag curvature ${\bf K}\leq 0$ is Riemannian (if ${\bf K}<0$) or locally Minkowskian (if ${\bf K}=0$). \underline{}Every isotropic Berwald metric is an isotropic Landsberg metric. Then  the Corollary \ref{cor1} is a generalization of Bao-Chern-Shen's result for homogeneous Finsler metrics.

\section{Proof of Theorem \ref{MainTHM1.5}}
 A Killing frame for an $n$-dimensional Finsler manifold $(M, F)$  is a set of local vector fields $\{X_i\}_{i =1}^n$, defined on an open subset $U\subseteq M$ around a given point, such that: (1)  The set of tangent vectors $\{X_i(x)|\,\,i=1, \ldots, n\}$ gives a basis for every tangent space $T_x(M)$, at any point  $x \in U$; (2) In $U$, each $ X_i $ satisfies $\tilde{X}_i(F) = 0$.

Though Killing frames are rare in the general study of Finsler geometry, they can be easily found for a homogeneous Finsler space at any given point \cite{HD2}. Let the homogeneous Finsler space $(M,F)$ be presented as $M=G/H$, where $H$ is the isotropy subgroup
for the given $x$. The tangent space $T_xM$ can be identified as the quotient
$\mathfrak{m}=\mathfrak{g}/\mathfrak{h}$,  where $\mathfrak{g}$ and $\mathfrak{h}$ are
the Lie algebras of $G$ and $H$,  respectively. Take any basis $\{ v_1,\ldots,v_n \}$ of $\mathfrak{m}$, with the pre-images $\{\hat{v}_1,\ldots,\hat{v}_n\}$ in $\mathfrak{g}$. Then the Killing vector fields $\{X_1,\ldots,X_n\}$ on $M$ corresponding to $\hat{v}_i$s defines  a Killing frame around $x$. The choice of $\hat{v}_i$s or $X_i$s identifies the quotient space $\mathfrak{m}$ with a subspace of $\mathfrak{g}$. Thus, we can write
the decomposition of linear space
\begin{equation}\label{decomp}
\mathfrak{g}=\mathfrak{h}+\mathfrak{m}.
\end{equation}
With respect to the decomposition \eqref{decomp}  there is a projection map ${\rm pr}:\mathfrak{g}\rightarrow\mathfrak{m}$. Note that for  the bracket operation $[\cdot,\cdot]$ on $\mathfrak{g}$ we have
$[\cdot,\cdot]_{\mathfrak{m}}={\rm pr}[\cdot,\cdot]$.

\bigskip

As mentioned before, Xu and Deng  proved that a homogeneous Finsler metric has isotropic S-curvature  if and only if it has vanishing S-curvature \cite{HD2}.   Here, we develop this fact to the E-curvature of homogeneous $(\alpha, \beta)$-metrics.
\begin{thm}\label{THM1}
Let $(G/H, F=\alpha\phi(\beta/\alpha))$ be an $n$-dimensional homogeneous $(\alpha, \beta)$-manifold, where $\alpha$ is a $G$-invariant Riemannian metric, and $\beta$ is a $G$-invariant 1-form on $G/H$.  Then the following are equivalent:
\begin{description}
 \item[] (i) $F$ has isotropic mean Berwald curvature ${\bf E}=(n+1)/2 cF^{-1}{\bf h}$;
 \item[] (ii)  $F$ has vanishing E-curvature ${\bf E}=0$;
\end{description}
where $c=c(x)$ is a scalar function on $G/H$.
\end{thm}
\begin{proof} Suppose that $\alpha=\sqrt{a_{ij}y^iy^j}$ is a Riemannian metric and $\beta=b_i(x)y^i$ is a 1-form on $M=G/H$.  Let $F:=\alpha\phi(s)$, $s=\beta/\alpha$, be an   $(\alpha, \beta)$-metric on the manifold $G/H$. Suppose that  $H$ is the isotropy subgroup
for the given $x \in M=G/H$. It suffices to prove the above equivalency at point $x$, since $(M,F)$ is a homogeneous Finsler manifold.

Here,  $\phi=\phi(s)$ is a $C^\infty$ function on the interval $(-b_0, b_0)$. For $b^2:=a^{ij}b_ib_j$, let
\begin{eqnarray}
\Phi:= -(n\Delta +1+sQ)(Q-sQ')-(b^2-s^2)(1 + sQ)Q'',
\end{eqnarray}
where
\be
Q:=\frac{\phi'}{\phi-s\phi'}, \ \ \ \  \Delta= 1+sQ+(b^2-s^2)Q'.
\ee
In \cite{DW}, Deng-Wang proved that the S-curvature of a homogeneous $(\alpha, \beta)$-metric, with respect to the decomposition (\ref{decomp}), is given by
\begin{equation}\label{DengWang}
{\bf S}(x,y)=\frac{1}{\alpha(x, y)}\frac{\Phi}{2\Delta^2}\Big (\kappa\langle [u, y]_\mathfrak{m}, y  \rangle+\alpha(x, y)Q\langle [u, y]_\mathfrak{m}, u  \rangle\Big ), \,\,\,\,\forall y\in \mathfrak{m},
\end{equation}
where $\kappa$ is a real constant and $u$ is the unique vector in $\mathfrak{m}$ corresponding to $\beta$.

It follows from (\ref{DengWang})  that
\[
{\bf S}(x, u)=0 \ \ \  \textrm{and} \ \ \  {\bf S}(x, -u)=0.
\]
Now, suppose that $F$ has isotropic mean Berwald curvature
\[
{\bf E}=\frac{n+1}{2F} c\ {\bf h},
\]
where $c=c(x)$ is a scalar function on $G/H$. Then, we  conclude that S-curvature of $F$ is in the following form
\begin{equation}\label{S-alaki1}
{\bf S}=(n+1)cF+\eta,
\end{equation}
where $\eta$ is 1-form on $G/H$. Putting $y=u$ and $y=-u$ in (\ref{S-alaki1}), respectively, we get
\begin{equation}\label{S-alaki2}
c(x)F(x, u)+\eta(x, u)=0,\,\,\,\, c(x)F(x, -u)-\eta(x, u)=0.
\end{equation}
Therefore,
\[
c(x)\big(F(x, u)+F(x, -u)\big)=0,
\]
and consequently $c(x)=0$. Putting it in \eqref{S-alaki1} implies that  ${\bf S}=\eta$ is a 1-form. Taking twice vertical derivative of ${\bf S}=\eta$ implies that the
E-curvature vanishes.
\end{proof}

\bigskip
Let  $F:=\alpha\phi(s)$, where $s=\beta/\alpha$ be an $(\alpha, \beta)$-metric  on a manifold $M$. Let us define
\be
\Xi:=\frac{(b^2 Q+s) \Phi}{\Delta^2}.\label{Xi0}
\ee
In \cite{NTa}, the authors proved the following.
\begin{thm}\label{NT}{\rm (\cite{NTa})}
Let $F=\alpha\phi(s)$, $s={\beta}/{\alpha}$, be an $(\alpha,\beta)$-metric. Suppose that $\Xi$ is not constant. Then $F$ is of isotropic $S$-curvature if and only if it is of isotropic $E$-curvature.
\end{thm}

\bigskip
Here, we prove that $\Xi$ is always non-constant for regular $(\alpha, \beta)$-metrics. 
\begin{thm}\label{THM2}
Let $F=\alpha\phi(s)$, $s={\beta}/{\alpha}$, be a regular $(\alpha,\beta)$-metric. Then $F$ is of isotropic $S$-curvature if and only if it is of isotropic $E$-curvature.
\end{thm}
\begin{proof}
First we remark that if $\Xi=0$ then we get
\[
(b^2 Q+s) \Phi=0.
\]
If $b^2 Q+s=0$, then we have $\phi=\sqrt{s^2-b^2}$ which is not regular. Thus  $\Phi=0$. In \cite{CWW}, it is proved that an $(\alpha,\beta)$-metric is a Rientannian metric if and only if $\Phi=0$. Thus $\Xi=0$ characterizes Riemannian metrics in the class of $(\alpha,\beta)$-metrics.

Now, we are going to show that for regular $(\alpha,\beta)$-metrics the quantity $\Xi$ is not constant. On contrary, suppose that $\Xi=c (constan)$. Then by \eqref{Xi0}, we have
\[
(b^2 Q+s) \Phi=c\Delta^2
\]
which yields
\be
\big[(b^2-s^2)\phi'+s\phi\big]\Phi=c (\phi-s\phi')\Delta^2.\label{Xi1}
\ee
By the regularity of $F=\alpha\phi(s)$, $b:=||\beta||_\alpha$  and $|s|\leq b$, we have
\[
\phi(s)>0,\ \ \ \ \ \phi(s)-s\phi'(s)>0, \ \  \forall |s|<b_0
\]
Letting $s$ approximate
\[
\frac{\phi(b)+\sqrt{\phi^2(b)+4b^2\phi'^2(b)}}{2\phi'(b)}
\]
in \eqref{Xi1} yields $c=0$. In this case, $F$ is Riemannian. Thus for regular Finsler metrics, $\Xi$ is not constant.
\end{proof}

\bigskip

By Theorems \ref{THM1} and \ref{THM2}, we conclude the following.
\begin{cor}
Let $(G/H, F=\alpha\phi(\beta/\alpha))$ be an $n$-dimensional homogeneous $(\alpha, \beta)$-manifold, where $\alpha$ is a $G$-invariant Riemannian metric, and $\beta$ is a $G$-invariant 1-form on $G/H$.  Then the following are equivalent:
\begin{description}
 \item[] (i) $F$ has isotropic mean Berwald curvature ${\bf E}=(n+1)/2 cF^{-1}{\bf h}$;
 \item[] (ii)  $F$ has isotropic S-curvature ${\bf S}=(n+1)cF$;
 \item[] (ii)  $F$ has isotropic S-curvature ${\bf S}=(n+1)cF+\eta$;
\end{description}
where $c=c(x)$ is a scalar function  and $\eta=\eta_i(x)y^i$  is a 1-form on $G/H$. In this case, ${\bf S}=0$.
\end{cor}
\bigskip

\noindent{\bf Proof of Theorem \ref{MainTHM1.5}:} Let  $F:=\alpha\phi(s)$, where $s=\beta/\alpha$ be an $(\alpha, \beta)$-metric  on a manifold $M=G/H$ introduced in Theorem \ref{THM1}. Let us define $b_{i;j}$ by $b_{i;j}\theta^j:=db_i-b_j\theta^j_{\ i}$, where
$\theta^i:=dx^i$ and $\theta^j_{\ i}:=\Gamma^j_{\ ik}dx^k$ denote
the Levi-Civita connection form of $\alpha$. Let us define
\begin{eqnarray*}
r_{ij}:=\frac{1}{2}\left(b_{i;j}+b_{j;i}\right), \ \ \ s_{ij}:=\frac{1}{2}\left(b_{i;j}-b_{j;i}\right),\ \ \ s_i:=b^j s_{ji},
\end{eqnarray*}
where $b^i:=a^{im}b_m$. By definition, the 1-form $\beta$ is parallel if and only if $r_{ij}=s_{ij}=0$.\\\\
Now, we consider two cases as follows:\\\\
{\bf Case (i): $F$ is non-Randers type metric.}  In \cite{C}, Cheng proved that every $(\alpha,\beta)$-metric $F:=\alpha\phi(s)$, $s=\beta/\alpha$, of non-Randers type, i.e., $\phi\neq c_1\sqrt{1+c_2s^2}+c_3s$ for any constants $c_1>0$, $c_2$ and $c_3$, has isotropic S-curvature ${\bf S}=(n+1)cF$ if and only if $\beta$ satisfies
\be
r_{ij} = 0, \ \ \ s_i= 0.\label{S}
\ee
In this case, ${\bf S}=0$, regardless of the choice of a particular $\phi=\phi(s)$. In \cite{LS} Li-Shen showed that an $(\alpha,\beta)$-metric $F:=\alpha\phi(s)$, $s=\beta/\alpha$, of non-Randers type on a manifold $M$ of dimension $n\geq 3$ has vanishing J-curvature ${\bf J}=0$, if and only if $\beta$ satisfies
\be
r_{ij} = k(b^2a_{ij} - b_ib_j), \ \ \ s_{ij}= 0.\label{JJ}
\ee
where $k = k(x)$ is a number depending on $x$ and $\phi=\phi(s)$ satisfies
\[
\Phi= \frac{\lambda}{\sqrt{b^2-s^2}}\Delta^{\frac{3}{2}},
\]
where $\lambda$ is a constant. By Lemma \ref{CLem}, \eqref{S}, and \eqref{JJ}, it follows that $\beta$ is parallel. In \cite{Sh}, Shen proved that a regular $(\alpha,\beta)$-metric is a Berwald metric if and only if its related 1-form $\beta$ is parallel. Then $F$ is a Berwald metric. Now, we have two subcases as follows:\\\\
{\bf Case (i)a: ${\bf K} =0$. } By the Numata theorem, every Berwald metric with zero flag curvature is locally Minkowskian.\\\\
{\bf Case (i)b: ${\bf K} <0$. } Suppose that ${\bf K} <0$ at a point $x\in M$. From \eqref{JIK*} and the assumption ${\bf J}=0$, we get $R^i_{\ m}I^mI_i=0$ which implies $I_i(x,y) =0$ for all $y\in T_xM_0$. By Deicke's theorem, $F$ is Riemannian.\\\\
{\bf Case (ii): $F$ is a Randers-type metric.} Now, suppose that $F$ is a Randers-type metric. It is well-known that every Randers-type metric is C-reducible
\be
{\bf C}_y(u, v, w)= {1\over n+1}\Big\{{\bf I}_y(u){\bf h}_y(v, w)+{\bf I}_y(v){\bf h}_y(u, w)+{\bf I}_y(w){\bf h}_y(u, v) \Big\}.\label{CR}
\ee
Taking a horizontal derivation of \eqref{CR} yields
\be
{\bf L}_y(u, v, w)= {1\over n+1}\Big\{{\bf J}_y(u){\bf h}_y(v, w)+{\bf J}_y(v){\bf h}_y(u, w)+{\bf J}_y(w){\bf h}_y(u, v) \Big\}.\label{LR}
\ee
By considering Lemma \ref{CLem}, the relation \eqref{LR} implies that ${\bf L}=0$. On the other hand, we have ${\bf S}=0$, which yields ${\bf E}=0$. In \cite{Cr}, Crampin showed that every Landsberg metric with vanishing mean Berwald curvature is a Berwald metric. This completes the proof.
\qed

\bigskip
By considering Case (i)b in Theorem \ref{MainTHM1.5}, we can conclude the following rigidity result that has been proved by Deng-Hu in \cite{DH3}.
\begin{cor}{\rm (\cite{DH3})}
Let $(M, F)$ be a homogeneous Randers space of Berwald type. If the
flag curvature of $F$ is negative everywhere, then $F$ is a Riemannian metric.
\end{cor}

\bigskip

Now, we consider homogeneous isotropic Berwald  metrics and prove the following.
\begin{cor}
Every  homogeneous isotropic Berwald  metric on a manifold $M$ of dimension $n\geq 3$  with non-positive flag curvature is Riemannian or locally Minkowskian.
\end{cor}
\begin{proof}
Every isotropic Berwald metric has isotropic mean Berwald curvature. In \cite{TN2}, the authors proved that every  homogeneous isotropic Berwald  metric on a manifold $M$ of dimension $n\geq 3$ is a Berwald metric or Randers metric of Berwald-type. By Numata theorem every Berwald metric of vanishing flag curvature f ${\bf K}= 0$ is locally Minkowskian. Also, if ${\bf K}\neq 0$, then by the same method used in  Case (i)b of Theorem \ref{MainTHM1.5}, $F$ is Riemannian. This completes the proof.
\end{proof}

\bigskip

The Douglas metrics are an extension of Berwald metrics introduced
by Douglas as a projective invariant class of  Finsler metrics. A Finsler
metric is called a Douglas
metric if
\[
G^i=\frac{1}{2}\Gamma^i_{jk}(x)y^jy^k + P(x,y)y^i,\label{Douglas0}
\]
where $\Gamma^i_{jk}=\Gamma^i_{jk}(x)$  are scalar functions on $M$ and
$P=P(x, y)$ is a homogeneous function of degree one with respect to
$y$ on $TM_0$.  For homogeneous Douglas metrics,  we prove the
following.
\begin{cor}
Every homogeneous Douglas $(\alpha, \beta)$-metric with non-positive
flag curvature is Riemannian or
locally Minkowskian.
\end{cor}
\begin{proof}
In \cite{LD}, Liu-Deng proved that every homogeneous $(\alpha,
\beta)$-metric is a Douglas metric if and only if  it is a Berwald
metric or  a Douglas metric of Randers type. By the Numata theorem for
Berwald metrics and Theorem \ref{MainTHM1.5}, we get the proof.
\end{proof}

\bigskip

We deal with spherically symmetric Finsler metrics. A Finsler metric $F=F(x, y)$ on a domain $\Omega\subseteq \mathbb{R}^n$ is called spherically symmetric metric if it is invariant under any rotations in $\mathbb{R}^n$ (see \cite{TBS}). It is proved that a Finsler metric $F$ on a convex domain $\Omega\subseteq \mathbb{R}^n$ is spherically symmetric if and only if there exists a positive function $\phi:=\phi(r,u,v)$, such that $F(x,y)=\phi(|x|,|y|,\langle x,y\rangle)$, where 
\[
|x|=\sqrt{ \sum^n_{i=1}(x^i)^2},\ \ \ |y|=\sqrt{\sum^n_{i=1}(y^i)^2}, \ \ \ \langle x,y\rangle=\sum^n_{i=1}x^iy^i.
\]
Here, we prove the following.
\begin{cor}
Every homogeneous spherically symmetric Finsler metric with non-positive flag curvature and isotropic Berwald curvature is Riemannian or locally Minkowskian.
\end{cor}
\begin{proof}
In \cite{GLM}, it is proved that every spherically symmetric Finsler metric $F=\phi(|x|,|y|,\langle x,y\rangle)$ on the $n$-ball $\mathbb{B}^n(r)$ with isotropic Berwald curvature is a Randers metric. By Theorem \ref{MainTHM1.5} we get the proof.
\end{proof}

\bigskip

\noindent
Behzad Najafi\\
Department of Mathematics and Computer Sciences\\
Amirkabir University (Tehran Polytechnic)\\
Hafez Ave.\\
Tehran. Iran\\
Email:\ behzad.najafi@aut.ac.ir

\bigskip

\noindent
Akbar Tayebi\\
Department of Mathematics, Faculty of Science\\
University of Qom \\
Qom. Iran\\
Email:\ akbar.tayebi@gmail.com

\end{document}